\RequirePackage{etex}
\documentclass[10pt,letterpaper]{amsart}

\usepackage{amsmath,amsthm,amsfonts,amssymb}
\usepackage{mathtools} %extension of amsmath, provides bug fixes and other tools 
\usepackage{mathrsfs} %Provides a nice \mathscr command
\usepackage[pdftex,bookmarks=true]{hyperref} %adds hypertext links, can't use in ``draft'' mode, and must use in pdf mode
\usepackage[thinlines]{easybmat}
\usepackage{amsrefs}
\usepackage[all,cmtip]{xy}

\newcommand{\mbb}[1]{\ensuremath \mathbb{#1}}
\newcommand{\mbf}[1]{\ensuremath \mathbf{#1}}
\newcommand{\mcl}[1]{\ensuremath \mathcal{#1}}
\newcommand{\msc}[1]{\ensuremath \mathscr{#1}}
\newcommand{\mrm}[1]{\ensuremath \mathrm{#1}}

\newcommand{\R}{\mathbb{R}}
\newcommand{\Z}{\mathbb{Z}}

\DeclareMathOperator{\SL}{SL}

\DeclareMathOperator{\SO}{SO}
\DeclareMathOperator{\Nil}{Nil}
\DeclareMathOperator{\id}{id}
\DeclareMathOperator{\Rm}{Rm}
\DeclareMathOperator{\Rc}{Rc}
\DeclareMathOperator{\scalar}{scal}

\DeclareMathOperator{\inj}{inj}
\DeclareMathOperator{\tr}{tr}
\DeclareMathOperator{\can}{can}
\DeclareMathOperator{\diam}{diam}
\DeclareMathOperator{\Hess}{Hess}

\newcommand{\fn}[3]{\nolinebreak{#1 \colon #2 \rightarrow #3}}
\newcommand{\fnl}[3]{\nolinebreak{#1 \colon #2 \longrightarrow #3}}
\newcommand{\fnla}[3]{#1 \colon #2 &\longrightarrow #3} %for proper alignment

\newcommand{\map}[2]{\nolinebreak{#1 \rightarrow #2}}
\newcommand{\mapl}[2]{#1 \longrightarrow #2}

\newcommand{\mapeltsla}[2]{#1 &\longmapsto #2}

\newcommand{\fndef}[5]{\begin{align*}\fnla{#1}{#2}{#3} \\ \mapeltsla{#4}{#5} \end{align*}}

\newcommand{\setv}[2]{\nolinebreak{ \{#1\hphantom{,}|\hphantom{,}#2\}}}

\newcommand{\setvbigl}[2]{\nolinebreak{ \left\{ \left. #1\hphantom{,} \right| \hphantom{,}#2 \right\}}}

\providecommand{\Scal}[1]{\left\langle #1\right\rangle}%        

\newcommand{\smfrac}[2]{{\textstyle{\frac{#1}{#2}}}}

\numberwithin{equation}{section}

\theoremstyle{plain} %% This is the default, anyway
\newtheorem{thm}[equation]{Theorem}

\newtheorem{lem}[equation]{Lemma}
\newtheorem{prop}[equation]{Proposition}

\theoremstyle{definition}
\newtheorem{defn}[equation]{Definition}

\theoremstyle{remark}
\newtheorem{rem}[equation]{Remark}
\newtheorem{ex}[equation]{Example}

\begin{document}

\author[M.~B.~Williams]{Michael Bradford Williams}
\address{Department of Mathematics, The University of California, Los Angeles}
\email{mwilliams@math.ucla.edu}
\urladdr{http://math.ucla.edu/~mwilliams}

\title{Results on coupled {R}icci and harmonic map flows}
%\date{\today}

\begin{abstract} 
We explore the harmonic-Ricci flow---that is, Ricci flow coupled with harmonic map flow---both as it arises naturally in certain principal bundle constructions related to Ricci flow and as a geometric flow in its own right.  We demonstrate that one natural geometric context for the flow is a special case of the locally $\mathbb{R}^N$-invariant Ricci flow of Lott, and provide examples of gradient solitons for the flow.  We prove a version of Hamilton's compactness theorem for the  flow, and then generalize it to the category of \'{e}tale Riemannian groupoids.  Finally, we provide a detailed example of solutions to the flow on the Lie group $\Nil^3$.
\end{abstract}

%%% AMS subject classification
%53C25    	Special Riemannian manifolds (Einstein, Sasakian, etc.)
%53C44    	Geometric evolution equations (mean curvature flow, Ricci flow, etc.)

%%%-------------------------------------------------------------------
%%% body of the paper
%%%-------------------------------------------------------------------

\maketitle
\tableofcontents

\section{Introduction}\label{sec:intro}

As the Ricci flow preserves isometry and holonomy groups \cites{Kotschwar2010,Kotschwar2011}, spaces with extra symmetries or additional structure are important in the study of this flow.  Such spaces include homogeneous and symmetric spaces, including Lie groups with left-invariant metrics, K\"{a}hler and Sasakian manifolds, warped products and rotationally symmetric spaces, and various types of vector and principal bundles.  In recent years, a number of authors have considered Ricci flow on manifolds with certain types of additional structure, with the result being a coupled system that involves Ricci flow and other geometric flows.  For example, Streets and Young independently discovered that Ricci flow on a principal bundle, with the assumption of fixed fiber volume, results in a coupling of Ricci flow on the base with Yang-Mills flow for the bundle connection \cites{Streets2007, Young2008}.  If $(\mcl{M},g)$ is a principal $\mcl{P}$-bundle with curvature 1-form $a$ and curvature $F$, then the flow is
\begin{equation}\label{eq:rym}
\begin{aligned}
\partial_t g &= -2 \Rc + F^2 \\
\partial_t a &= -d^* F
\end{aligned}
\end{equation}
Lott and Sesum studied Ricci flow on warped products with $\mcl{S}^1$ fibers, which (after modification by diffeomorphisms) becomes Ricci flow on the base coupled with heat flow for the warping function \cite{LottSesum2011}.  If $(\mcl{M} \times \mcl{S}^1,g+e^{2u} d\theta^2)$ is such a warped product, then the flow is
\begin{equation}\label{eq:wprf}
\begin{aligned}
\partial_t g &= -2 \Rc + 2 \, du \otimes du \\
\partial_t u &= \Delta u
\end{aligned}
\end{equation}
One can also consider a fiber of arbitrary dimension $m$.  Under certain curvature restrictions, the equations are the same, except with a factor of $m$ on the coupling term.  See \cites{Tran2012,Williams2012}.

Lott has also studied Ricci flow on a certain class of ``twisted'' principal $\mbb{R}^N$-bundles, where the flow becomes Ricci flow on the base, a Yang-Mills-type flow for the connection, and a heat-type flow for the fiber metrics \cite{Lott2010}.  Certain metrics on such bundles $\mbb{R}^N \hookrightarrow \mcl{M}^{n+N} \rightarrow \mcl{B}^n$ can be represented locally as $(g,A,G)$, where $g$ is a metric on the base, $A$ is an $\R^N$-valued $1$-form corresponding to a connection on $\mcl{M}$, and $G$ is an inner product on the fibers.  Ricci flow for these \textit{locally $\R^N$-invariant} metrics is
\begin{equation}\label{eq:inv-RF}
\begin{aligned}
\partial_t g_{\alpha\beta} &= -2R_{\alpha\beta} + \frac{1}{2} G^{ik} G^{j\ell} \nabla_{\alpha} G_{ij} \nabla_{\beta} G_{k\ell} 
+ g^{\gamma\delta} G_{ij} (dA)_{\alpha\gamma}^{i} (dA)_{\beta\delta}^{j},\\
\partial_t A_{\alpha}^{i} &= -(\delta dA)_{\alpha}^{i} + G^{ij} \nabla^{\beta} G_{jk} (dA)_{\beta\alpha}^{k} \\
\partial_t G_{ij} &= \Delta G_{ij} - G^{k\ell} \nabla_{\alpha} G_{ik} \nabla^{\alpha} G_{\ell j} - \frac{1}{2} g^{\alpha\gamma} 
g^{\beta\delta} G_{ik} G_{j\ell} (dA)_{\alpha\beta}^{k} (dA)_{\gamma\delta}^{\ell}.
\end{aligned}
\end{equation}

On the other hand, one might consider the coupling of Ricci flow with other flows, perhaps without a specific geometric motivation.  For example, List studied the coupling of Ricci flow with heat flow for functions in order to address certain questions in general relativity \cite{List2008}.  If $(\mcl{M},g)$ is a manifold and $u:\mcl{M}\rightarrow \mbb{R}$ is a function, then the flow is
\begin{equation}\label{eq:hrf-1d}
\begin{aligned}
\partial_t g &= -2 \Rc + 2 \smfrac{n-1}{n-2} \, du \otimes du \\
\partial_t u &= \Delta u
\end{aligned}
\end{equation}
Taking this a step further, M\"uller recently studied the abstract coupling of Ricci flow and harmonic map flow \cite{Muller2012}.  More precisely, if $\phi : (\mcl{M},g) \rightarrow (\mcl{N},h)$ is a map of Riemannian manifolds, the \textit{harmonic-Ricci flow} is the coupled system
\begin{equation}\label{eq:hrf}
\begin{aligned}
\partial_t g    &= -2 \Rc + 2 c \, \nabla \phi \otimes \nabla \phi \\
\partial_t \phi &= \tau_{g,h} \phi 
\end{aligned}
\end{equation}
where $\tau_{g,h} \phi$ is the harmonic map Laplacian (or tension field) of $\phi$, $\nabla \phi \otimes \nabla \phi = \phi^* h$, and $c$ is a (possibly time-dependent) coupling constant.

As described in \cite{Muller2012}, this flow has several desirable properties, and can actually be more well-behaved than either individual flow.  For example, unlike the harmonic map flow, many results do not require restrictions on the curvature of the target manifold.  It also arises as the gradient flow of a certain energy functional in the same way that Ricci flow is the gradient flow of an energy functional \cite{Perelman2002}.  M\"{u}ller also proved versions of a number of foundational theorems that are crucial to any analysis of the flow: solutions have short-time existence and uniqueness, derivatives of curvature of solutions satisfy Bernstein-Bando-Shi-type estimates, and long-time existence is obstructed only by the norm of the curvature tensor.

In this paper, we continue the investigation of the harmonic-Ricci flow.  First, in Section \ref{subsec:hrf-ex} we describe a new connection between the flows \eqref{eq:inv-RF} and \eqref{eq:hrf}, which shows that harmonic-Ricci flow can arise in certain geometric contexts\footnote{Ricci flow and the harmonic map flow are indeed known to be related.  For example, the DeTurck trick modifies Ricci flow by diffeomorphisms to make it strictly parabolic, and was used to prove short-time existence and uniqueness of Ricci flow on closed manifolds in \cite{DeTurck1983}.  Hamilton subsequently observed that these diffeomorphisms actually solve a modified harmonic map flow \cite{Hamilton1993}.  The phenomenon that we describe is particular to locally $\R^N$-invariant Ricci flow, and is independent of any harmonic map flow connection from a DeTurck trick.}.  Essentially, the heat-type equation for $G$ is a modified harmonic map flow, when $G$ is thought of as a map from the base $\mcl{B}$ into the space of inner products on the $\R^N$-fibers, $\SL(N,\R)/\SO(N)$.  See Proposition \ref{prop:hmf-G}.  This connection is best seen in the special case of a flat $\R^N$-vector bundle, which Lott used in order to study expanding Ricci solitons \cite{Lott2007}.  Here, Ricci flow is precisely the coupled flow \eqref{eq:hrf}.  In Section \ref{subsec:gr-sol} we consider gradient solitons for the harmonic-Ricci flow, and list a few examples.

In Section \ref{sec:compactness} is the main result of this paper.  It is a basic analytical tool for the flow, an adaptation of the ``compactness theorem'' of Hamilton \cite{Hamilton1995}, which allows one to extract convergent subsequences from sequences of solutions to the flow, assuming certain bounds on curvature and injectivity radius.  Such theorems are useful in the investigation of flows near singular times.  See Theorem \ref{thm:compactness}.  Following Lott \cite{Lott2007}, in Section \ref{sec:grpd} we generalize this compactness theorem to the setting of Riemannian groupoids, where the bound on injectivity radius is not needed.  As demonstrated in \cites{Glick2008,Lott2007}, groupoids provide a unified way to discuss convergence, for example, of Riemannian manifolds under geometric flows when the limit is a different manifold.  

Finally, in Section \ref{sec:ex} we provide a detailed, nontrivial example of harmonic-Ricci flow solutions on the Lie group $\Nil^3$, where the metrics are left-invariant and the map is a harmonic real-valued function.  The behavior of these solutions depends strongly on the coupling function, although it is similar to that of Ricci flow solutions if the function decays fast enough as $t \rightarrow \infty$.

\section{The harmonic-Ricci flow}\label{sec:hrf}

Let us provide background for the coupled flow in question.  Let $(\mcl{M},g)$ be a closed Riemannian manifold, with $(\mcl{N},h)$ a closed target manifold.  Let $\fn{\phi}{\mcl{M}}{\mcl{N}}$ be a smooth map.  The Levi-Civita covariant derivative $\nabla^{T\mcl{M}}$ of the metric $g$ on $\mcl{M}$ induces a covariant derivative $\nabla^{T^*\mcl{M}}$ on the cotangent bundle, which satisfies
\[ \nabla^{T^*\mcl{M}}_X \omega (Y) = X \big( \omega(Y) \big) - \omega \left( \nabla^{T\mcl{M}}_X Y \right). \]
By requiring a product rule and compatibility with the metric, we also have convariant derivatives on all tensor bundles 
\[ T^p_q(\mcl{M}) = (T^*\mcl{M})^{\otimes p} \otimes (T\mcl{M})^{\otimes q}. \]
The Levi-Civita covariant derivative $\nabla^{T\mcl{N}}$ of the metric $h$ on $N$ induces a covariant derivative $\nabla^{\phi^* T\mcl{N}}$ on the pull-back bundle $\phi^* T\mcl{N} \rightarrow \mcl{M}$, given by
\[ \nabla^{\phi^* T\mcl{N}}_X Y = \nabla^{T\mcl{N}}_{\phi_* X} Y, \]
for $X \in C^\infty(T \mcl{M})$ and $Y \in C^\infty(T \mcl{N})$.  As before, we get a covariant derivative on all tensor bundles over $\mcl{M}$ of the form
\[ T^p_q(\mcl{M}) \otimes T^r_s(\phi^* \mcl{N}) = (T^*\mcl{M})^{\otimes p} \otimes (T\mcl{M})^{\otimes q} \otimes (\phi^* T^*\mcl{N})^{\otimes r} \otimes (\phi^* T\mcl{N})^{\otimes s}. \]
We refer to them simply as $\nabla$.  Related quantities are decorated with the metric name, if necessary, e.g., ${}^g \nabla$.  In local coordinates $(x^i)$ on $\mcl{M}$ and $(y^\lambda)$ on $\mcl{N}$,
\[ \nabla \phi = \phi_* = \partial_i \phi^{\lambda} \, dx^i \otimes \partial_{\lambda}|_{\phi} \in C^\infty(T^*\mcl{M} \otimes \phi^* T\mcl{N}). \]
Similarly, we have
\begin{align*}
\nabla^2 \phi 
&= \big( \partial_i \partial_j \phi^\lambda - {}^g \Gamma_{ij}^k \partial_k \phi^\lambda + ({}^h \Gamma \circ \phi)_{\mu \nu}^\lambda \partial_i \phi^\mu \partial_j \phi^\nu \big) \, dx^i \otimes dx^j \otimes \partial_\lambda|_\phi \\
&\in C^\infty(T^*\mcl{M} \otimes T^*\mcl{M} \otimes \phi^* T\mcl{N}). 
\end{align*}
Additionally
\[ \nabla \phi \otimes \nabla \phi = h_{\lambda \mu} \partial_i \phi^\lambda \partial_j \phi^\mu \, dx^i \otimes dx^j = \phi^* h\]
and is a symmetric $(2,0)$-tensor on $\mcl{M}$, and we define
\[ \mrm{S} = \Rc - c \, \nabla \phi \otimes \nabla \phi, \]
so that
\[ \mrm{s} = \tr \mrm{S} = \scalar - c |\nabla \phi|^2, \]
where $c=c(t) \geq 0$ is a coupling function.  Finally, the \textit{tension field} of $\phi$ with respect to $g$ and $h$ is
\begin{equation}\label{eq:tension}
\tau_{g,h} \phi = \tr_g \nabla^2 \phi \in C^{\infty}(\phi^* TN),
\end{equation}
and has components
\[ (\tau_{g,h} \phi)^\lambda = g^{ij} \Big( \partial_i \partial_j \phi^\lambda - {}^g \Gamma_{ij}^k \partial_k \phi^\lambda + ({}^h \Gamma \circ \phi)_{\mu \nu}^\lambda \partial_i \phi^\mu \partial_j \phi^\nu \Big). \]

Now we recall the flow \eqref{eq:hrf}.  

\begin{defn}\label{def:hrf}
If $\phi : (\mcl{M},g) \rightarrow (\mcl{N},h)$ is a map of Riemannian manifolds, the \textit{harmonic-Ricci flow} is the coupled system
\begin{equation}\tag{\ref{eq:hrf}}
\begin{aligned}
\partial_t g    &= -2 \mrm{S} = -2 \Rc + 2 c \, \nabla \phi \otimes \nabla \phi \\
\partial_t \phi &= \tau_{g,h} \phi 
\end{aligned}
\end{equation}
\end{defn}

We will call this \textsc{hrf} for short, although it is also sometimes called the $(RH)_c$ flow.  We will assume that $c(t)$ is non-increasing.  As mentioned above, this flow was introduced in \cite{Muller2012} and is a generalization of one studied in \cite{List2008}.  

\begin{defn}\label{def:rh-soln}
A family $\{(\mcl{M}^n,g(t),\phi(t),O)\}$ of complete, pointed Riemannian manifolds with maps 
\[ \fnl{\phi(t)}{\mcl{M}}{\mcl{N}} \]
that solves the system \eqref{eq:hrf} with coupling function $c(t)$, for $t \in (\alpha,\omega)$, is a \textit{complete, pointed \textsc{hrf} solution}.
\end{defn}

\subsection{A connection with Ricci flow on certain principal bundles}\label{subsec:hrf-ex}

In this section, we explain a relationship between harmonic-Ricci flow \eqref{eq:hrf} and locally $\mbb{R}^N$-invariant Ricci flow \eqref{eq:inv-RF}, showing that harmonic-Ricci flow arises naturally in a certain geometric context.  Let $\mcl{M} \rightarrow \mcl{B}$ be a twisted principal $\mbb{R}^N$-bundle over a compact base and let $(g,A,G)$ be a locally $\mbb{R}^N$-invariant metric, see \cite{Lott2010}.  Write $\mcl{S}_N = \SL(N,\R)/\SO(N)$ for the space of inner products on $\mbb{R}^N$ with fixed volume.  The tangent space $T_G \mcl{S}_N$ at $G \in \mcl{S}_N$ consists of trace-free symmetric bilinear forms.  There is a Riemannian metric on $T_G \mcl{S}_N$ defined by
\begin{equation}\label{eq:S-metric}
\overline{g}_G(X,Y) = \tr(G^{-1} X G^{-1} Y) = G^{ij} X_{jk} G^{k\ell} Y_{\ell i}.
\end{equation}
%The tension field of $\fn{G}{\mcl{B}}{\mcl{S}_N}$, with respect to the metrics $g$ and $\overline{g}$, has components
%\begin{equation}\label{eq:tau-G}
%(\tau_{g,\overline{g}} G)_{ij}
%= \Delta G_{ij} + g^{\alpha \beta} \sum_{\substack{p<q\\r<s}}^N ({}^{\mcl{S}_N} \! \Gamma \circ G)_{pq,rs}^{ij} \nabla_\alpha G_{pq} \nabla_\beta G_{rs}.
%\end{equation}
%The reader is invited to compare this definition to the general formulation in \eqref{eq:tension}.

\begin{prop}\label{prop:hmf-G}
The evolution equation for $G$ from \eqref{eq:inv-RF} is a modified harmonic map flow for $\fn{G}{\mcl{B}}{\mcl{S}_N}$.  More precisely,
\[ \frac{\partial}{\partial t} G_{ij} = (\tau_{g(t),\overline{g}} G)_{ij}
- \frac{1}{2} g^{\alpha\gamma} g^{\beta\delta} G_{ik} G_{j\ell} (dA)_{\alpha\beta}^{k}(dA)_{\gamma\delta}^{\ell}. \]
\end{prop}

\begin{proof}
What we are really claiming is that 
\begin{equation}\label{eq:tension-simple}
\Delta G_{ij} - g^{\alpha\beta} G^{k\ell} \nabla_\alpha G_{ik} \nabla_\beta G_{\ell j} 
= (\tau_{g,\overline{g}} G)_{ij}.
\end{equation}
The map $G$ has energy
\[ E(G) = \frac{1}{2} \int_B g^{\alpha \beta} \tr( G^{-1} \nabla_\alpha G^{-1} \nabla_\beta G) \, dV. \]
In \cite{Lott2007}*{Proposition 4.17} it is shown that the variational equation of this energy is 
\[ \Delta G_{ij} - g^{\alpha\beta} G^{k\ell} \nabla_\alpha G_{ik} \nabla_\beta G_{\ell j} = 0. \]
It follows from the variational definition of the tension field that the left side of this equation must be $\tau_{g(t),\overline{g}}(G)$. %from \eqref{eq:tau-G}.
\end{proof}

%\begin{rem}
%It is possible to verify equation \eqref{eq:tension-simple} directly with a (lengthy) computation.
%\end{rem}

%\begin{ex}\label{ex:expanding}
In \cite{Lott2007} Lott considers a special case of locally $\mbb{R}^N$-invariant Ricci flow \eqref{eq:inv-RF}.  Namely, let $\mcl{M}$ be an $\R^N$-vector bundle with flat connection preserving fiberwise volume forms, flat metric $G$ on the fibers, and Riemannian base $(\mcl{B},g)$.  Write the metric on $\mcl{M}$ as $(g,0,G)$.  Then the fiber metrics constitute a map $G:\mcl{B} \rightarrow \mcl{S}_N$ as before.  From \cite{Lott2007}*{Equation (4.10)}, Ricci flow on $\mcl{M}$ becomes the pair of equations
\begin{align*}
\frac{\partial}{\partial t} g_{\alpha\beta} &= -2 R_{\alpha\beta} + \frac{1}{2} G^{ij} \nabla_\alpha G_{jk} G^{k\ell} \nabla_\beta G_{\ell i} \\
\frac{\partial}{\partial t} G_{ij}          &=  g^{\alpha\beta} \nabla_\alpha \nabla_\beta G_{ij} - g^{\alpha\beta} \nabla_\alpha G_{ik} G^{kl} \nabla_\beta G_{\ell i}.
\end{align*}
But with the metric $\overline{g}$ on $\mcl{S}_N$ as in \eqref{eq:S-metric}, we see that
\[ \frac{1}{2} G^{ij} \nabla_\alpha G_{jk} G^{k\ell} \nabla_\beta G_{\ell i} = \frac{1}{4} (\nabla G \otimes \nabla G)_{\alpha\beta}, \]
and Proposition \ref{prop:hmf-G} says that
\[ g^{\alpha\beta} \nabla_\alpha \nabla_\beta G_{ij} - g^{\alpha\beta} \nabla_\alpha G_{ik} G^{kl} \nabla_\beta G_{\ell i} = (\tau_{g,\overline{g}} G)_{ij}. \]
This means Ricci flow on $\mcl{M}$ is precisely \textsc{hrf} on $\mcl{B}$, with target manifold $(\mcl{S}_N,\overline{g})$, maps $G$, and $c=1/8$.

This gives many examples of \textsc{hrf} solutions.  For instance, the homogeneous spaces in \cite{Lott2007} that admit expanding Ricci solitons all have the bundle structure just described.  Those solitons correspond to Ricci flow solutions, and hence to \textsc{hrf} solutions.
%\end{ex}

\subsection{Gradient harmonic-Ricci solitons}\label{subsec:gr-sol}

Metrics leading to self-similar solutions of Ricci flow---that is, Ricci solitons---have been studied extensively in recent years.  In this section we briefly consider the corresponding pairs $(g,\phi)$ that give self-similar solutions of harmonic-Ricci flow.   Following \cite{Muller2012}*{Definition 2.1}, given $(\mcl{M},g)$, $(\mcl{N},h)$, and $\varphi : \mcl{M} \rightarrow \mcl{N}$, a \textit{harmonic-Ricci soliton} is a pair $(g,\phi)$ such that
\begin{align*}
g(t)    &= \sigma(t) \varphi_t^* g \\
\phi(t) &= \varphi_t^* \phi
\end{align*}
solves \eqref{eq:hrf}, for some function $\sigma$ and some family of diffeomorphisms $\varphi_t$ of $\mcl{M}$.  Suppose that these diffeomorphisms are generated by the gradient of a function $f$.  By \cite{Muller2012}*{Lemma 2.2}, the soliton condition is then equivalent to the following static equations for $g$, $\phi$, and $f$:
\begin{equation}\label{eq:gsol}
\begin{aligned}
0 &= \Rc - \alpha \, \nabla \phi \otimes \nabla \phi + \Hess(f) + \lambda g \\
0 &= \tau_{g,h} \phi - \nabla \phi(\nabla f)
\end{aligned}
\end{equation}
Denote a gradient soliton by $(g,\phi,f)$.  The soliton is \textit{shrinking/steady/expanding} for $\lambda$ positive/zero/negative.  As with Ricci flow, there is a functional for which \eqref{eq:hrf} is the gradient flow, and for which solitons are critical points.

As we mentioned in Section \ref{subsec:hrf-ex}, Lott considers Ricci flow on a flat $\R^N$-vector bundle $\mcl{M}$ over a manifold $(\mcl{B},g)$, with flat metrics $G$ on the fibers, and such that the fiberwise volume forms are preserved by the flat connection \cite{Lott2007}.  This construction arises as a model for non-compact expanding Ricci solitons.  Lott derived the soliton equation in this case and showed that it is a soliton-like equation for the base, plus the condition that the metric $G$, interpreted as a map $\mcl{B} \rightarrow \mcl{S}_N$, is harmonic \cite{Lott2007}*{Theorem 1.2 or Proposition 4.4}.  He calls those the \textit{harmonic-Einstein} equations.  We generalize that notion in the following definition.

\begin{defn}
Given $(\mcl{M},g)$, $(\mcl{N},h)$, and $\varphi : \mcl{M} \rightarrow \mcl{N}$, the pair $(g,\phi)$ is \textit{harmonic-Einstein} if
\begin{equation}\label{eq:heinstein}
\begin{aligned}
0 &= \Rc - \alpha \, \nabla \phi \otimes \nabla \phi + \lambda g \\
0 &= \tau_{g,h} \phi
\end{aligned}
\end{equation}
for some $\alpha,\lambda \in \mbb{R}$.
\end{defn}

In the same way that an Einstein metric is a fixed point of normalized Ricci flow, a harmonic-Einstein pair is a fixed point of the analogous normalized harmonic-Ricci flow.  We also have an extension of a result of Hamilton that holds for Ricci solitons, which rules out steady or expanding gradient solitons where $f$ is non-trivial \cite{Hamilton1995-sing}.

\begin{prop}
If $\mcl{M}$ is closed and $(g,\phi,f)$ is a steady or expanding gradient harmonic-Ricci soliton, then $(g,\phi)$ is harmonic-Einstein. 
\end{prop}

\begin{proof}
By \cite{Muller2012}*{Theorem 4.4}, we have that under \textsc{hrf}
\[ \partial_t \mrm{s} = \Delta \mrm{s} + 2|\mrm{S}|^2 + 2c |\tau_{g,h}|^2. \]
On the other hand, we can compute $\partial_t \mrm{s}$ from the equation $\mrm{s}[\epsilon \varphi^* g] = \epsilon^{-1} \varphi^* (s[g])$, and combine it with the above expression to obtain
\[ \Delta \mrm{s} + 2|\mrm{S}|^2 + 2c |\tau_{g,h}|^2 - \mcl{L}_X \mrm{s} + \epsilon \mrm{s} = 0, \]
where $\epsilon \in \mbb{R}$ and $\varphi$ is a diffeomorphism generated by a vector field $X$.  Now, the remainder of the proof virtually the same as the proof of \cite{Chow2007}*{Proposition 1.13}, with $\mrm{S}$ in place of $\Rc$ and $\mrm{s}$ in place of $\scalar$. 
\end{proof}

\begin{ex} The ubiquity of both gradient Ricci solitons and harmonic maps in Riemannian geometry leads to many examples harmonic-Ricci solitons. 

\begin{enumerate}
\item Let $(\mcl{M},g,f)$ be any gradient Ricci soliton, let $(\mcl{N},h)$ be any Riemannian manifold, and let $\phi : \mcl{M} \rightarrow \mcl{N}$ be a constant map.  Then $(g,\phi,f)$ is a gradient harmonic-Ricci soliton, which we will call \textit{trivial}.

\item A harmonic-Einstein pair $(g,\phi)$ is a harmonic-Ricci soliton in the same way an Einstein metric is a Ricci soliton with $f$ trivial.  Examples of these pairs are found in \cite{Lott2007}; see Section \ref{subsec:hrf-ex} above.

\item Let $(\mcl{N},h) = (\mcl{M},g)$, let $\phi$ be a $g$-isometry, let $f$ be constant, and let $g$ be $\mu$-Einstein.  Then $\phi$ is automatically harmonic, so the second equation in \eqref{eq:gsol} is satisfied.  The first equation is satisfied if and only if $\mu - \alpha + \lambda = 0$. 

A slightly more specific example of this is when $\mcl{M}=\mcl{G}$ is a Lie group, $g$ is a left-invariant metric, and $\phi = L_x$ is left multiplication by any element of $x \in \mcl{G}$.

\item To generalize the previous example, let $\phi : (\mcl{M},g) \hookrightarrow (\mcl{N},h)$ be a minimal isometric immersion (so $\phi$ is harmonic), where $(\mcl{M},g)$ is $\mu$-Einstein and $f$ is constant.  Again, this is a soliton if and only if $\mu - \alpha + \lambda = 0$.

\item Consider a Riemannian submersion $\pi : (\mcl{M},\tilde{g}) \rightarrow (\mcl{B},g)$ such that the $A$ and $T$ tensors vanish (see \cite{Besse2008}*{Chapter 9}).  In particular, $\mcl{M}$ is locally a Riemannian product of $(\mcl{B},g)$ and $(\mcl{F},h)$.  The map $\pi$ is harmonic since the $T$-tensor vanishes .  If $h$ is $\mu$-Einstein and $(\mcl{B},g,f)$ is a gradient soliton with soliton constant $\lambda$, then we see that \eqref{eq:gsol} is satisfied if and only if $\mu - \alpha + \lambda = 0$.
\end{enumerate}
\end{ex}

\begin{rem}
In all of these examples, either $\phi$ is harmonic and $f$ is trivial, or $\phi$ is trivial and $(g,f)$ is a gradient Ricci soliton.  It would be interesting to find examples of gradient harmonic-Ricci solitons where this is not the case.
\end{rem}

\section{A compactness theorem}\label{sec:compactness}

So-called compactness theorems are crucial in the study of geometric flows, especially regarding singularity models.  For example, one often wishes to construct sequences of rescaled solutions of a flow in order to investigate the behavior at a singular time (possibly $T = +\infty$), and it is helpful to be able to extract convergent subsequences.  In the context of Ricci flow, the original compactness theorem appears in \cite{Hamilton1995}, and was generalized to the groupoid setting in \cite{Lott2007}.  A version for \textsc{hrf}, with 1-dimensional target, appears in \cite{List2008}.  First we need a definition.

\begin{defn}\label{def:rh-converge}
A sequence $\{(\mcl{M}_k^n,g_k(t),\phi_k(t),O_k)\}$ of complete, pointed \textsc{hrf} solutions \textit{converges} to $(\mcl{M}_{\infty}^n,g_{\infty}(t),\phi_{\infty}(t),O_{\infty})$ for $t \in (\alpha,\omega)$ if there exists
\begin{itemize}
\item an exhaustion $\{\mcl{U}_k\}$ of $\mcl{M}_{\infty}$ by open sets with $O_{\infty} \in \mcl{U}_k$ for all $k$, and
\item a family of pointed diffeomorphisms $\{ \fn{\Psi_k}{(\mcl{U}_k,O_\infty)}{(\mcl{V}_k,O_k) \subset \mcl{M}_k} \}$ 
\end{itemize}
such that
\[ \left( \mcl{U}_k \times (\alpha,\omega), \Psi_k^* \left( g_k(t)|_{\mcl{V}_k} + dt^2 \right), \Psi_k^* \phi_k|_{\mcl{V}_k} \right) \]
converges uniformly in $C^{\infty}$ on compact sets to
\[ \left( \mcl{M}_{\infty}^n \times (\alpha,\omega), g_{\infty}(t) + dt^2, \phi_{\infty}(t) \right). \]
Here, $dt^2$ is the standard metric on $(\alpha,\omega) \subset \R$.
\end{defn}

\begin{thm}\label{thm:compactness}
Let $\{(\mcl{M}^n_k, g_k(t), \phi_k(t), O_k)\}$ be a sequence of complete, point\-ed \textsc{hrf} solutions, with $0,t \in (\alpha, \omega)$, $c(t)$ non-increasing, and $(\mcl{N},h)$ compact, such that
\begin{itemize}
\item[(a)] the geometry is uniformly bounded: for all $k$,
\[ \sup_{(x,t) \in \mcl{M}_k \times (\alpha, \omega)} |\Rm_k|_k \leq C \]
for some $C_1$ independent of $k$;
%\item[(b)] the maps are controlled: for all $k$,
%\[ \sup_{(x,t) \in \mcl{M}_k \times (\alpha, \omega)} |\phi_k|_h \leq C_2 \]
%for some $C_2$ independent of $k$;
\item[(b)] the initial injectivity radii are uniformly bounded below: for all $k$, 
\[ \inj_{g_k(0)}(O_k) \geq \iota_0 > 0, \]
for some $\iota_0$ independent of $k$.
\end{itemize}
Then there is a subsequence such that 
\[ \mapl{\big( \mcl{M}_k,g_k(t),\phi_k(t),O_k \big)}{\big( \mcl{M}_{\infty},g_{\infty}(t),\phi_{\infty}(t),O_{\infty} \big)}, \]
where the limit is also a pointed, complete, \textsc{hrf} solution.

If we do not assume a bound on he injectivity radius bound, then we have convergence to 
\[ \big( \mcl{G}_{\infty},g_{\infty}(t),\phi_{\infty}(t),O_{\infty} \big), \]
a complete, pointed, $n$-dimensional, \'{e}tale Riemannian groupoid with map $\phi_{\infty}$ on the base.
\end{thm}

\begin{rem}
It is somewhat surprising that, unlike many theorems involving simply the harmonic map flow, this theorem makes no assumptions about the curvature of the codomain manifold.  Therefore, the only difference between the statment of this theorem and the statement of Hamilton's original theorem is the flow for which we are considering solutions.
\end{rem}

The idea of the proof is the same as in \cite{Hamilton1995}, and subsequently \cite{List2008}, although we follow the exposition found in \cite{Chow2007}*{Chapter 3}.  Briefly, the main ingredients are derivative estimates to bound the curvature and the derivatives of the map $\phi$, a general compactness theorem of Hamilton, a technical lemma, and corollary of the Arzela-Ascoli theorem.  Of course, many facts about \textsc{hrf}, found in \cite{Muller2012}, are used along the way.

\begin{ex}\label{ex:blowdown}
Here is a way to obtain sequences of \textsc{hrf} solutions like those considered in the compactness theorem.  Consider a solution of \textsc{hrf} with $c(t)$ constant: $\big(g(t),\phi(t))$.  We can obtain a family of \textsc{hrf} solutions by performing a \textit{blowdown}, a technique used extensively in \cite{Lott2007} and \cite{Lott2010}, and also in \cites{LottSesum2011,Williams2010a}.  For $s \in (0,\infty)$, define 
\[ \big(g_s(t),\phi_s(t)\big) = \left( \frac{1}{s} g(st), \phi(st) \right). \]
Now we see that
\[ \partial_t g_s(t) = \left( \partial_t g \right)(st) \quad \text{and} \quad \partial_t \phi_s(t) = s \left( \partial_t \phi \right)(st), \]
and
\begin{align*}
\mrm{S}[g_s(t),\phi_s(t))] 
&= -2 \Rc[g_s(t)] + 2 c \, \nabla \phi_s(t) \otimes \nabla \phi_s(t) \\
&= -2 \Rc[g(st)] + 2 c \, \nabla \phi(st) \otimes \nabla \phi(st) \\
&= \mrm{S}[g(st),\phi(st)],
\end{align*}
\begin{align*}
\tau_{g_s(t),h} \phi_s(t)
&= \tr_{g_s(t)} \nabla^2 \phi_s(t) \\
&= s \tr_{g(st)} \nabla^2 \phi(st) \\
&= s \tau_{g(st),h} \phi(st).
\end{align*}
Therefore, for each $s$, the blowdown gives another \textsc{hrf} solution.  It is common to replace the continuous parameter $s$ with a sequence $\{s_j\}$ converging to infinity.  

Given the success of this technique regarding Ricci flow solutions \cites{Lott2007,Lott2010}, it seems likely that this technique will be useful for studying solutions of the harmonic-Ricci flow.
\end{ex}

We mention that we will use abbreviated notation for geometric objects associated with metrics $g_k(t)$.  For example, $\Rc_k(t)$ means $\Rc[g_k(t)]$, and $\nabla_k$ refers to the Levi-Civita covariant derivative corresponding to the metric $g_k(t)$ (not to be confused with the covariant derivative $\nabla_{\partial/\partial_k}$---it should be clear from context which is intended).  Also, an undecorated $\nabla$ will refer to the Levi-Civita covariant derivative corresponding to a background metric.

\subsection{Two lemmas}\label{subsec:cpt-lemma}

In this section we prove two lemmas that will be used in the proof of Theorem \ref{thm:compactness}.  The first is an analogue of \cite{Hamilton1995}*{Lemma 2.4}, \cite{Chow2007}*{Lemma 3.11}, and \cite{List2008}*{Lemma 7.6}.

\begin{lem}\label{lem:tech}
Let $(\mcl{M}^n,g)$ be a Riemannian manifold, with $\mcl{K} \subset \mcl{M}$ compact.  Let $\{(g_k(t),\phi_k(t))\}$ be a sequence of solutions to \textsc{hrf}, defined on $\mcl{K} \times [\beta,\psi]$, where $t_0 \in [\beta,\psi]$.  Suppose the following hold.
\begin{itemize}
\item[] The metrics $g_k(t_0)$ are uniformly equivalent to $g$ on $\mcl{K}$.  That is, for all $x \in \mcl{K}$, $V \in T_x \mcl{M}$, and $k$, there is $C < \infty$ such that
\begin{equation}\label{met-equiv}
C^{-1} g(V,V) \leq g_k(t_0)(V,V) \leq C g(V,V). 
\end{equation}

\item[] The covariant derivatives of $g_k(t_0)$ and $\phi_k$ with respect to $g$ are uniformly bounded on $K$.  That is, for all $p \geq 0$, there exist $C_p, C_p'$ such that
\begin{equation}\label{cov-met-bd}
\max_{x \in \mcl{K}} |\nabla^{p+1} g_k(t_0)| \leq C_p < \infty,
\end{equation}
\begin{equation}\label{cov-map-bd}
\max_{x \in \mcl{K}} |\nabla^p \phi_k(t_0)| \leq C_p' < \infty.
\end{equation}

\item[] The covariant derivatives of $\Rm_k$ and $\phi_k$ with respect to $g_k(t)$ are uniformly bounded on $K \times [\beta,\psi]$.  That is, for all $p \geq 0$, there exist $C_p'', C_p'''$ such that
\begin{equation}\label{cov-rm-bd}
\max_{x \in \mcl{K}} |\nabla^p_k \Rm_k|_k \leq C_p'' < \infty,
\end{equation}
\begin{equation}\label{cov-d-bd}
\max_{x \in \mcl{K}} |\nabla^{p}_k \phi_k|_k \leq C_p''' < \infty.
\end{equation}
\end{itemize}
Then the following hold.
\begin{itemize}
\item[] The metrics $g_k(t)$ are uniformly equivalent to $g$ on $\mcl{K} \times [\beta,\psi]$.  That is, for all $x \in \mcl{K}$, $V \in T_x \mcl{M}$, $k$, there exists $B > 0$ such that
\begin{equation}\label{met-equiv-t}
B^{-1} g(V,V) \leq g_k(t)(V,V) \leq B g(V,V).
\end{equation}

\item[] The time and covariant derivatives with respect to $g$ of $g_k(t)$ and $\phi_k(t)$ are uniformly bounded on $\mcl{K} \times [\beta,\psi]$.  That is, for all $p$ and $q$, there exist $\tilde{C}_{p,q}, \tilde{D}_{p,q}$ such that
\begin{equation}\label{time-cov-met-bd}
\max_{x \in \mcl{K}} \left| \frac{\partial^q}{\partial t^q} \nabla^p g_k(t) \right| \leq \tilde{C}_{p,q} < \infty,
\end{equation}
\begin{equation}\label{time-cov-map-bd}
\max_{x \in \mcl{K}} \left| \frac{\partial^q}{\partial t^q} \nabla^p \phi_k(t) \right| \leq \tilde{D}_{p,q} < \infty.
\end{equation}
\end{itemize}
\end{lem}

\begin{proof} First, note that throughout the proof we will follow standard practice in not indexing constants, and will often use the same symbol (e.g., $C$) for different constants within a sequence of inequalities.

To prove (a), we have
\[ \partial_t g_k(t) = -2 \mrm{S}_k(t) = -2 \Rc_k(t) + 2 c(t) \nabla_k \phi_k(t) \otimes \nabla_k \phi_k(t), \]
so that for $V \in T\mcl{M}$,
\begin{align*}
\left| \partial_t g_k(t)(V,V) \right|
&= |-2 \Rc_k(t)(V,V) + 2 c(t) \nabla_k \phi_k(t) \otimes \nabla_k \phi_k(t)(V,V)| \\
&\leq 2 |\Rc_k(t)||V|_k^2 + 2 |c(t)| |\nabla_k \phi_k(t)|^2 |V|_k^2 \\
&\leq C' |V|_k^2 \\
&= C' g_k(t)(V,V).
\end{align*}
This implies
\[ |\partial_t \log g_k(t)(V,V)| = \left| \frac{\partial_t g_k(t)(V,V)}{g_k(t)(V,V)} \right| \leq C', \]
and thus for any $t_1 \in [\beta,\psi]$, we have
\[ \int_{t_0}^{t_1} |\partial_t \log g_k(t)(V,V)| \, dt \leq C' |t_1 - t_0|. \]
This gives
\begin{align*}
C' |t_1 - t_0|
&\geq \int_{t_0}^{t_1} |\partial_t \log g_k(t)(V,V)| \, dt \\
&\geq \left| \int_{t_0}^{t_1} \partial_t \log g_k(t)(V,V) \, dt \right| \\
&= \left| \log \frac{g_k(t_1)(V,V)}{g_k(t_0)(V,V)} \right|.
\end{align*}
Expanding this gives
\[ -C'|t_1-t_0| \leq \log \frac{g_k(t_1)(V,V)}{g_k(t_0)(V,V)} \leq C'|t_1-t_0|, \]
and exponentiating gives
\[ \exp(-C'|t_1-t_0|) g_k(t_0)(V,V) \leq g_k(t_1) \leq \exp(C'|t_1-t_0|) g_k(t_0)(V,V). \]
Combining this with the original hypotheses, we get
\[ C^{-1} \exp(-C'|t_1-t_0|) g(V,V) \leq g_k(t_1) \leq C \exp(C'|t_1-t_0|) g(V,V). \]
Since $t_1$ was arbitrary, and since $C\exp(C'|t_1-t_0|) \leq C\exp(C'|\psi-\beta|) = B$, this completes the proof of (a).

Next, we prove \eqref{time-cov-met-bd} and \eqref{time-cov-map-bd}.  Observe that
\begin{equation}\label{der-metric}
\left| \frac{\partial^q}{\partial t^q} \nabla^p g_k(t) \right| = \left| \nabla^p \frac{\partial^{q-1}}{\partial t^{q-1}} \frac{\partial}{\partial t} g_k(t) \right| = 2 \left| \nabla^p \frac{\partial^{q-1}}{\partial t^{q-1}} \mcl{S}_k(t) \right|,
\end{equation} 
\begin{equation}\label{der-map} \left| \frac{\partial^q}{\partial t^q} \nabla^p \phi_k(t) \right| = \left| \nabla^p  \frac{\partial^{q-1}}{\partial t^{q-1}} \frac{\partial}{\partial t} \phi_k(t) \right| = \left| \nabla^p \frac{\partial^{q-1}}{\partial t^{q-1}} \tau_{g_k} \phi_k(t) \right|.
\end{equation}
Recall that $\nabla$ is the Levi-Civita covariant derivative corresponding to the background metric $g$.  In general,
\[ \begin{aligned}
S_{ij} &= R_{ij} - c \nabla_i \phi \nabla_j \phi, \\   
(\tau_{g,h} \phi)^\lambda &= g^{ij} \nabla_i \nabla_j \phi^\lambda, % = g^{ij}( \partial_i \partial_j \phi^\lambda - \Gamma_{ij}^k \partial_k \phi^\lambda + ({}^N \! \Gamma_{\mu \nu}^\lambda \circ \phi) \partial_i \phi^\mu \partial_j \phi^\nu),   
  \end{aligned}
\]
and we have the following evolution equations for $\mcl{S}$ and $\tau_{g,h} \phi$:  
\[ \begin{aligned}
\partial_t S_{ij} &= 
\Delta_\ell S_{ij} + 2 c \, \tau_{g,h} \phi \, \nabla_i \nabla_j \phi - \dot{c} \nabla_i \phi \nabla_j \phi \\
\partial_t \tau_{g,h} \phi
&= - g^{ik} g^{jl} S_{kl} (\nabla_i \nabla_j \phi) 
+ g^{ij} \Big( \Delta(\nabla_i \phi \nabla_j \phi) - 2 \nabla_p \nabla_i \phi \nabla_p \nabla_j \phi \\
&\quad - R_{ip} \nabla_p \phi \nabla_j \phi - R_{jp} \nabla_p \phi \nabla_i \phi 
+ 2 \Scal{{}^\mcl{N} \! \Rm(\nabla_i \phi,\nabla_p \phi) \nabla_p \phi,\nabla_j \phi} \Big).
\end{aligned} \]
To bound \eqref{der-metric} and \eqref{der-map}, we need to consider the evolution equations for all quantities involved, which appear in \cite{Muller2012}:
%\begin{itemize}
%\item Christoffel symbols:
\[ \begin{aligned}   
\partial_t \Gamma_{ij}^p 
&= - g^{pq} (\nabla_i R_{jq} + \nabla_j R_{iq} - \nabla_q R_{ij}
   - 2c \nabla_i \nabla_j \phi \nabla_q \phi) 
\end{aligned} \]
%\item Riemann:
\[ \begin{aligned}
\partial_t R_{ijk\ell} 
&= \nabla_i \nabla_k R_{j\ell} - \nabla_i \nabla_\ell R_{jk}
-\nabla_j \nabla_k R_{i\ell} + \nabla_j \nabla_\ell R_{ik} - R_{ijq\ell} R_{kq} \\ 
&\quad - R_{ijkq} R_{\ell q} + 2c \big(\nabla_i \nabla_k \phi \nabla_j \nabla_\ell \phi -\nabla_i \nabla_\ell \phi \nabla_j \nabla_k \phi \big) \\
&\quad - 2c \Scal{ {}^\mcl{N} \! \Rm(\nabla_i \phi,\nabla_j \phi) \nabla_k \phi, \nabla_\ell \phi}.
\end{aligned} \]
%\item Ricci:
\[ \begin{aligned}
\partial_t R_{ij} 
&= \Delta_\ell R_{ij} - 2R_{iq} R_{jq} + 2R_{ipjq} R_{pq}
+ 2c \, \tau_{g,h} \phi \nabla_i \nabla_j \phi - 2c \nabla_p \nabla_i \phi \nabla_p \nabla_j \phi\\
&\quad \, + 2c R_{pijq} \nabla_p \phi \nabla_q \phi + 2c
\Scal{{}^\mcl{N} \! \Rm (\nabla_i\phi,\nabla_p\phi) \nabla_p \phi, \nabla_j \phi}. 
\end{aligned} \]
%\end{itemize}
In these equations, we used 
\[ \Scal{ {}^\mcl{N} \! \Rm (\nabla_i\phi,\nabla_j\phi) \nabla_j\phi, \nabla_i\phi} 
:= {}^\mcl{N} \! R_{\kappa\mu\lambda\nu}\nabla_i \phi^\kappa\nabla_j\phi^\mu\nabla_i\phi^\lambda\nabla_j\phi^\nu, \]
and $k$ was a coordinate index, not a sequence index.

The types of terms that will appear in the expansions of \eqref{der-metric} and \eqref{der-map} therefore involve factors containing
\begin{equation}\label{eq:geom-quants}
\mrm{S}_k, \Rc_k, \Rm_k, \nabla_k \phi_k, \tau_{g_k,h} \phi_k, {}^\mcl{N} \! \Rm, 
\end{equation}
as well as time and covariant derivatives, whose norms we must show are bounded.  Note that we can ingore the geometric factors coming from the manifold $\mcl{N}$, since those quantities are bounded by compactness of $\mcl{N}$ and by the chain rule. %\marginpar{is this right?}

Now, let us consider the case $p=1,q=0$ for \eqref{time-cov-met-bd} and \eqref{time-cov-map-bd}.  As in the proof of Lemma 3.11 in \cite{Chow2007}, we have
\begin{equation}\label{gamma-ineq}
\frac{1}{2} |\nabla g_k(t)|_k \leq |\Gamma_k - \Gamma|_k \leq \frac{3}{2} |\nabla g_k(t)|_k.
\end{equation}
That is, up to lowering/raising indices, the tensors $\nabla g_k(t)$ and $\Gamma_k - \Gamma$ are equivalent.  Using the evolution of the Christoffel symbols, an estimation in normal coodinates gives
\[ \left| \partial_t(\Gamma_k - \Gamma) \right|_k^2 
\leq 12|\nabla_k \Rc_k|_k^2 + 8 \overline{c} |\nabla_k^2 \phi_k|_k^2 |\nabla_k \phi_k|_k^2 \\
\leq C. \]
We can show that $|\Gamma_k - \Gamma|_k$ is bounded by integrating the above inequality:
\begin{align*}
C|t_1-t_0|
&\geq \int_{t_0}^{t_1} | \partial_t (\Gamma_k(t) - \Gamma)|_k \, dt \\
&\geq \left| \int_{t_0}^{t_1} \partial_t (\Gamma_k(t) - \Gamma) \, dt \right|_k \\
&\geq | \Gamma_k(t_1) - \Gamma|_k - |\Gamma_k(t_0) - \Gamma|_k.
\end{align*}
Since $t_1$ is arbitrary, we see that
\begin{align*}
|\Gamma_k(t) - \Gamma|_k 
&\leq C|t-t_0| + |\Gamma_k(t_0) - \Gamma|_k \\
&\leq C|t-t_0| + \frac{3}{2} |\nabla g_k(t_0)|_k \\
&\leq C|t-t_0| + \frac{3}{2} B |\nabla g_k(t_0)| \\
&\leq C.
\end{align*}
From this and \eqref{met-equiv-t} it follows that
\begin{align*}
|\nabla g_k(t)| 
&\leq C |\nabla g_k(t)|_k \\
&\leq C |\Gamma_k(t) - \Gamma|_k \\
&\leq C,
\end{align*}
and we also have
\begin{align*}
|\nabla \phi_k(t)| 
&\leq C |\nabla \phi_k(t)|_k \\
&\leq C \left( |(\nabla - \nabla_k) \phi_k(t)|_k + |\nabla_k \phi_k(t)|_k \right) \\
&\leq C \left( |\Gamma_k(t) - \Gamma|_k |\phi_k(t)| + |\nabla_k \phi_k(t)|_k \right) \\
&\leq C.
\end{align*}
This completes the case for $p=1$, $q=0$. 

The general case will follow once we bound the norms of the quantities listed in \eqref{eq:geom-quants} and their deriviatives.  For this we need several preliminary bounds:
\begin{align}
|\nabla^p \mrm{S}_k(t)| &\leq C |\nabla^p g_k(t)| + C', \label{eq:nabla-s} \\
|\nabla^p \phi_k(t)|    &\leq C'', \label{eq:nabla-phi} \\
|\nabla^p g_k(t)|       &\leq C'''. \label{eq:nabla-g} 
\end{align}
We prove these by induction.  Consider \eqref{eq:nabla-s}.  Since $\mcl{S} = \Rc - c \nabla \phi \otimes \nabla \phi$, we have
\begin{align*}
|\mrm{S}_k|_k 
&= |\Rc_k - c \nabla_k \phi_k \otimes \nabla_k \phi_k |_k \\
&\leq |\Rc_k|_k + \overline{c}|\nabla_k \phi_k|_k^2 \\
&\leq C, 
\end{align*}
and
\begin{align*}
|\nabla_k \mrm{S}_k|_k
&= |\nabla_k \Rc_k - \nabla_k(c \nabla_k \phi_k \otimes \nabla_k \phi_k)|_k \\
&= |\nabla_k \Rc_k - \dot{c} \nabla_k \phi_k \otimes \nabla_k \phi_k - c \nabla_k (\nabla_k \phi_k \otimes \nabla_k \phi_k)|_k \\
&\leq |\nabla_k \Rc_k |_k + |\dot{c}| |\nabla_k \phi_k \otimes \nabla_k \phi_k|_k + 2 \overline{c} | \nabla_k^2 \phi_k \otimes \nabla_k \phi_k |_k \\
&\leq C_1' + |\dot{c}| |\nabla_k \phi_k|_k^2 + 2 \overline{c} |\nabla_k^2 \phi_k|_k |\nabla_k \phi_k |_k \\
&\leq C.
\end{align*}
Now, we can use this to see that
\begin{align*}
|\nabla \mrm{S}_k|
&\leq C |\nabla \mcl{S}_k|_k \\
&\leq C |(\nabla - \nabla_k) \mcl{S}_k|_k + B^{3/2} |\nabla_k \mcl{S}_k|_k \\
&\leq C |\Gamma_k - \Gamma|_k |\mcl{S}_k|_k + B^{3/2} |\nabla_k \mcl{S}_k|_k \\
&\leq C,
\end{align*}
so the base case is complete.

Assume that \eqref{eq:nabla-s} holds for all $p < N$, and then consider $p = N$, for $N \geq 2$.  Using the difference of powers formula, we have
\begin{align*}
|\nabla^N \mrm{S}_k|
&= \left| \sum_{i=1}^N \nabla^{N-i}(\nabla - \nabla_k) \nabla_k^{i-1} \mrm{S}_k + \nabla_k^N \mrm{S}_k \right| \\
&\leq \sum_{i=1}^N | \nabla^{N-i}(\nabla - \nabla_k) \nabla_k^{i-1} \mrm{S}_k| + |\nabla_k^N \mrm{S}_k|.
\end{align*}
The goal now is to show that we can bound $|\nabla^{N-i}(\nabla - \nabla_k) \nabla^{i-1}_k\mcl{S}_k|$.  Recall that $\nabla - \nabla_k = \Gamma - \Gamma_k$ is a sum of terms of the form $\nabla g_k$.  In what follows, we will informally write this as $\sum \nabla g_k$.  

Now, suppose $i=1$.  Then using the product rule repeatedly, we have
\begin{align*}
|\nabla^{N-1}(\nabla - \nabla_k) \mrm{S}_k|
&= |\nabla^{N-1}(\sum \nabla g_k) \mrm{S}_k| \\
&= \left| \sum_{j=0}^{N-1} \binom{N-1}{j} \nabla^{N-1-j}(\sum \nabla g_k) \nabla^j \mrm{S}_k \right| \\
&\leq \sum_{j=0}^{N-1} \binom{N-1}{j} \sum| \nabla^{N-j} g_k | |\nabla^j \mrm{S}_k |. 
\end{align*}
Each term here is bounded by inductive hypothsis.

Similarly, for $2 \leq i \leq N$, we have
\begin{align*}
|\nabla^{N-i}(\nabla - \nabla_k) \nabla_k^{i-1} \mrm{S}_k|
&\leq \sum_{j=0}^{N-i} \binom{N-i}{j} \sum| \nabla^{N-i-j+1} g_k | |\nabla^j \nabla_k^{i-1} \mrm{S}_k |. 
\end{align*}
We need to estimate the last factor.  In general we have
\begin{align*}
|\nabla^j \nabla_k^{i} \mrm{S}_k |
&= | [(\nabla - \nabla_k) + \nabla_k]^j \nabla_k^i \mrm{S}_k | \\
&= \left| \sum_{l=0}^j \binom{j}{l} (\nabla - \nabla_k)^{j-l} \nabla_k^l \nabla_k^i \mrm{S}_k \right| \\
&\leq \sum_{l=0}^j \binom{j}{l} |\nabla - \nabla_k|^{j-l} |\nabla_k^{l+i} \mrm{S}_k | \\
&\leq \sum_{l=0}^j \binom{j}{l} \sum |\nabla g_k|^{j-l} |\nabla_k^{l+i} \mrm{S}_k |.
\end{align*}
This is also bounded by inductive hypothesis.  Putting it all together (the assumptions of the lemma, the inductive hypotheses, equivalence of the norms) we have the desired bounds.  

The same method can be used to verify \eqref{eq:nabla-phi}.

For \eqref{eq:nabla-g}, we have
\[ \partial_t \nabla^N g_k(t)
= \nabla^N \frac{\partial}{\partial t} g_k(t) 
= -2 \nabla^N \mcl{S}_k(t). \]
This implies
\begin{align*}
\partial_t |\nabla^N g_k |^2
&= 2 \left\langle \partial_t \nabla^N g_k, \nabla^N g_k \right\rangle \\
&\leq \left| \partial_t \nabla^N g_k(t) \right|^2 + |\nabla^N g_k(t)|^2 \\
&= 4 |\nabla^N \mcl{S}_k|^2 + |\nabla^N g_k(t)|^2 \\
&\leq C|\nabla^N g_k|^2 + D
\end{align*}
We can integrate this differential inequality to get
\[ |\nabla^N g_k(t)|^2 \leq C, \]
as desired.

Using the arguments above, one can show that
\[ |\nabla^p \nabla_k^q \Rc_k|, |\nabla^p \nabla_k^q \Rm_k|, |\nabla^p \nabla_k^q R_k|, |\nabla^p \nabla_k^q \mrm{S}_k|, |\nabla^p \nabla_k^q \phi_k| \]
are bounded, independent of $k$. 

Finally, we note that $\tau_{g_k,h} \phi_k$ and its derivatives have bounded norm.  This follows from $\tau_{g,h} \phi = g^{ij} \nabla_i \nabla_j \phi$.  

All terms are thus bounded, and we conclude that (\ref{der-metric}) and (\ref{der-map}) are as well.
\end{proof}

The second lemma, which is a corollary of the Arzela-Ascoli theorem is a modification of \cite{Chow2007}*{Corollary 3.15}.

\begin{lem}\label{lem:aa}
Let $(\mcl{M}^n,g)$ be a Riemannian manifold, with $\mcl{K} \subset \mcl{M}$ compact and $p \in \Z^{\geq 0}$.  Suppose $\{(g_k,\phi_k)\}$ is a sequence of Riemannian metrics on $\mcl{K}$ and maps $\map{\mcl{K}}{\mcl{N}}$, where $\mcl{N}$ is some fixed target manifold, such that
\[ \sup_{0 \leq \alpha \leq p+1} \max_{x \in \mcl{K}} |\nabla^{\alpha} g_k| \leq C_1 < \infty, \]
\[ \sup_{0 \leq \alpha \leq p+1} \max_{x \in \mcl{K}} |\nabla^{\alpha} \phi_k| \leq C_2 < \infty. \]
Addionally, suppose that there exists $\delta > 0$ such that $|V|_k \geq \delta |V|$ for all $V \in T\mcl{M}$.  Then there exists a subsequence $\{(g_{k_j,}\phi_{k_j})\}$, a Riemannian metric $g_{\infty}$ on $\mcl{K}$, and a smooth map $\fn{\phi_{\infty}}{\mcl{K}}{\mcl{N}}$ such that $\map{(g_{k_j},\phi_{k_j})}{(g_{\infty},\phi_{\infty})}$ in $C^p$ as $\map{k}{\infty}$.
\end{lem}

\begin{proof}
The existence of the subsequence will follow from the Arzela-Ascoli theorem, so we need to show that the collection of component functions $\{ (g_k)_{ab} \} \cup \{ (\phi_k)^{\lambda} \}$ is an equibounded and equicontinuous family.  Equiboundedness follows from the hypotheses.  

Now, in a fixed coordinate chart, by writing
\[ \nabla_a (g_k)_{bc} = \partial_a (g_k)_{bc} - \Gamma_{ab}^d (g_k)_{dc} - \Gamma_{ac}^d (g_k)_{bd} \]
we see that bounds on $|\nabla g_k|$ give bounds on $|\partial_a (g_k)_{bc}|$.  Similarly,
\[ |\nabla_a (\phi_k)^{\lambda}| = |\partial_a (\phi_k)^{\lambda}| \]
is assumed to be bounded.  Now, the mean value theorem for functions of several variables implies that 
\[ |(g_k)_{bc}(y) - (g_k)_{bc}(x)| \leq C_1 \diam(K), \]
for all $x,y \in K$ and all indices $b,c$, and similarly for components of $\phi_k$.  This means the family $\{(g_k)_{ab}\} \cup \{(\phi_k)^{\lambda}\}$ is equicontinuous in the chart.  Since $K$ is compact, we can take finitely many charts to see that there is a finite uniform bound.  Now apply the Arzela-Ascoli theorem to obtain the limits $g_{\infty}$ and $\phi_{\infty}$.  The bounds on the metrics imply that $g_{\infty}$ is also a metric, and clearly $\phi_{\infty}$ is smooth.  

We have only demonstrated subsequential convergence in $C^0$.  For $C^p$ convergence, repeat the same arguments starting with covariant derivatives of $g_k$ and $\phi_k$, obtaining bounds on the higher partial derivatives.
\end{proof}

\subsection{The proof of the theorem}\label{subsec:cpt-proof}

We will need a result of Hamilton, Theorem 2.3 in \cite{Hamilton1995}, which he used to prove the original compactness theorem for Ricci flow.

\begin{thm}\label{thm:ham-conv}
Let $\{(\mcl{M}^n_k,g_k,O_k)\}$ be a sequence of pointed, complete, Riemannian manifolds such that
\begin{itemize}
\item[(a)] the geometry is uniformly bounded:
\[ |\nabla_k^p \Rm_k|_k \leq C_p \]
on $\mcl{M}_k$, for all $p \geq 0$, all $k$, for $C_p$ independent of $k$;
\item[(b)] the injectivity radii are uniformly bounded below:
\[ \inj_k(O_k) \geq \iota_0 \> 0, \]
for some $\iota_0$ independent of $k$.
\end{itemize}
Then there is a subsequence such that 
\[ \mapl{(\mcl{M}_k,g_k,O_k)}{(\mcl{M}_{\infty},g_{\infty},O_{\infty})}, \]
where the limit is also a pointed, complete, Riemannian manifold.
\end{thm}

We will also need the derivative estimate for the curvature and the map, Theorem 6.10 in \cite{Muller2012}.  This is a version of the Bernstein-Bando-Shi estimates for Ricci flow (see \cite{ChowKnopf2004}*{Section 7.1} for exposition).

\begin{thm}\label{thm:muller-bbs}
Let $(\mcl{M}^n,g(t),\phi(t))$ solve \textsc{hrf} for $t \in [0,\omega)$ and $c(t)$ non-in\-creasing.  Assume $0 < \underline{c} \leq c(t) \leq \overline{c} < \infty$ for all $t$, and that $\omega < \infty$.  Suppose that the curvature is uniformly bounded:
\[ \sup_{\mcl{M} \times [0,\omega)} |\Rm| \leq R_0. \]
Then there exists a constant $C = C(\underline{c},\overline{c},R_0,T,m,N) < \infty$ such that
\[ \sup_{\mcl{M} \times (0,\omega)} |\nabla \phi|^2 \leq \frac{C}{t}, \]
\[ \sup_{M \times (0,\omega)} \Big( |\Rm|^2 + |\nabla^2 \phi|^2 \Big) \leq \frac{C^2}{t^2}. \]
Moreover, there exist constants $C_p$ depending on $p, \overline{c}$, $m$ and $N$ such that
\[ \sup_{M \times (0,\omega)} \Big( |\nabla^p \Rm|^2 + |\nabla^{p+2} \phi|^2 \Big) \leq C_p \left( \frac{C}{t} \right)^{p+2}. \]
\end{thm}

Now we prove the theorem, in the presense of a bound on the injectivity radius.  The proof of the groupoid statement will appear in the next subsection.

\begin{proof}[Proof of Theorem \ref{thm:compactness}]
First, note that we may use a diagonalization argument, as in \cite{Hamilton1995}*{Section 2}, to show that we can assume that the interval of existence of the solutions is finite in length, that is, 
\[ -\infty < \alpha < \omega < \infty. \]

Since we are assuming that the curvatures are uniformly bounded, Theorem \ref{thm:muller-bbs} applies to give uniform bounds on the derivatives of the curvatures and on the derivatives of the maps $\phi_k$.  With the former, and with the injectivity radius bound, we can use Theorem \ref{thm:ham-conv} to get pointed subsequential convergence of the metrics at a single time, say $0 \in (\alpha,\omega)$:
\[ \map{\big( \mcl{M}_k,g_k(0),O_k \big)}{\big( \mcl{M}_{\infty},g_{\infty},O_{\infty} \big)}. \]
The limit is a complete, pointed Riemannian manifold.

Unpacking this convergence, we have the existence of
\begin{itemize}
\item an exhaustion $\{\mcl{U}_k\}$ of $\mcl{M}_{\infty}$ by open sets with $O_{\infty} \in \mcl{U}_k$ for all $k$, and
\item a family of pointed diffeomorphisms $\{ \fn{\Psi_k}{(\mcl{U}_k,O_\infty)}{(\mcl{V}_k,O_k) \subset \mcl{M}_k} \}$
\end{itemize}
such that
\[ \mapl{\big( \mcl{U}_k, \Psi_k^* g_k(0)|_{\mcl{V}_k} \big)}{\left( \mcl{M}_{\infty}^n, g_{\infty} \right)} \]
uniformly in $C^{\infty}$ on compact sets.  

We want to extend this convergence to all times $t \in [\alpha,\omega]$, and also obtain convergence of the maps to some limit $\phi_\infty(t)$.  To this end, the metrics and maps we now consider are $\bar{g}_k(t) = \Psi_k^* g_k(t)$ and $\bar{\phi}_k(t) = \Psi_k^* \phi_k(t)$.

Now we see that the hypotheses of the Lemma \ref{lem:tech} are satisfied.  For any compact $\mcl{K} \subset \mcl{M}_\infty$ and $[\beta,\psi] \subset [\alpha,\omega]$ containing $0$, the collection $\{(\bar{g}_k(t),\bar{\phi}_k(t))\}$ is a sequence of \textsc{hrf} solutions on $K \times [\beta,\psi]$.  Let $g_\infty$ be the background metric and $t_0 = 0$.

The uniform convergence implies that the $\bar{g}_k(0)$ are uniformly equivalent to $g_\infty$, and that the needed bounds hold.  For example, using the equivalence of metrics and convergence at one time, we see that
\begin{align*}
|\nabla_\infty^p \bar{\phi}_k(0)|_\infty
&\leq C |\nabla_\infty^p \bar{\phi}_k(0)|_{\bar{g}_k(0)} \\
&\leq C |\nabla_{\bar{g}_k(0)}^p \bar{\phi}_k(0)|_{\bar{g}_k(0)} \\
&\leq C |\nabla_{g_k(0)}^p \phi_k(0)|_{g_k(0)} \\
&\leq C,
\end{align*}
for large enough $k$.

By the lemma, we conclude that $\bar{g}_k(t)$ are uniformly equivalent to $g_\infty$ on $\mcl{K} \times [\beta,\psi]$, and that the time and space derivatives of $\bar{g}_k(t)$ and $\bar{\phi}_k(t)$ are uniformly bounded with respect to $g_\infty$.

Now, the conditions of Lemma \ref{lem:aa} are exactly satisfied by the implications of Lemma \ref{lem:tech}, so we have the desired subsequential convergence.  Our limit solution is defined by
\[ g_\infty(t) = \lim_{k \rightarrow \infty} \bar{g}_k(t), \quad 
\phi_\infty(t) = \lim_{k \rightarrow \infty} \bar{\phi}_k(t). \]

Finally, since all derivatives of the metric and the of map converge, the appropriate tensors converge, so that the limit is a metric/smooth map solving \textsc{hrf}.
\end{proof}

\section{The flow on groupoids}\label{sec:grpd}

Riemannian groupoids can be used to understand collapse of manifolds as in the context of geometric flows.  For example, Lott and Glickenstein use this framework to explain the collapsing behavior of Ricci flow on Lie groups in low dimensions \cites{Glick2008,Lott2007}.  Here we decribe how harmonic-Ricci flow can be defined in terms of groupoids, and then prove the last part of Theorem \ref{thm:compactness}.  For more information on groupoids, see the above papers, or, \cites{Mackenzie2005,Moerdijk2003} for example.

Recall that a Lie groupoid $\mcl{G} \rightrightarrows \mcl{B}$ is \textit{Riemannian} if the base $\mcl{B}$ has a $\mcl{G}$-invariant metric $g$.  That is, if $\mcl{U} \subset \mcl{B}$ is open, $\fn{\sigma}{\mcl{U}}{\mcl{G}}$ is any local bisection, and $\fn{t}{\mcl{G}}{\mcl{B}}$ is the target map, then $(t \circ \sigma)^*g = g$.  From this, we can construct the Ricci tensor $\Rc[g]$, which is a symmetric $(2,0)$-tensor on $\mcl{B}$, and which is $\mcl{G}$-invariant in the same sense as $g$.  Ricci flow on this groupoid is just what we expect:
\[ \frac{\partial}{\partial t} g = -2 \Rc. \]

Let $(\mcl{N},h)$ be another Riemannian manifold, thought of as a trivial groupoid, and consider $\fn{\phi}{\mcl{B}}{\mcl{N}}$ such that $\nabla \phi \otimes \nabla \phi$ is a $\mcl{G}$-invariant $(2,0)$-tensor on $\mcl{B}$.  Additionally, the tension field $\tau_{g,h} \phi$ of $\phi$ is well-defined in the usual Riemannian manifold sense.  Therefore, we have a well-defined coupling of Ricci flow and harmonic map flow:
\[ \begin{aligned}
\frac{\partial}{\partial t} g    &= -2 \Rc + 2 c \nabla \phi \otimes \nabla \phi \\
\frac{\partial}{\partial t} \phi &= \tau_{g,h} \phi
\end{aligned} \]
where $c(t)$ is a non-negative coupling function.

To use this approach to understand limits and convergence of \textsc{hrf} on Riemannian manifolds, we show how this groupoid setting can arise from the manifold setting.  Let $(\mcl{M},g)$ and $(\mcl{N},h)$ be complete Riemannian manifolds, and $\fn{\phi}{\mcl{M}}{\mcl{N}}$ a smooth map.  Select $\{p_i\}_{i \in I} \subset \mcl{M}$ such that $\msc{U} = \{\mcl{U}_i\}_{i \in I}$ is an open cover of $\mcl{M}$, where the $\mcl{U}_i$ are such that $\exp_{p_i}(0) = p_i \in \mcl{U}_i$, and 
\[ \fnl{\exp_{p_i}|_{B_{r_i}(0)}}{B_{r_i}(0)}{\mcl{U}_i} \]
is a diffeomorphism, for some sufficiently small $r_i>0$.  Put the metric $(\exp_{p_i})^* g$ on each $B_{r_i}(0)$.  Call $\msc{U}$ an \textit{open exponential cover} of $\mcl{M}$.

As in \cite{Lott2007}*{Example 5.7}, from this we form a Riemannian groupoid $\mcl{G}^\msc{U} \rightrightarrows \mcl{B}^\msc{U}$, which is isometrically equivalent to the trivial groupoid $(\mcl{M},g)$.  Set
\[ \mcl{B}^\msc{U} = \bigsqcup_{i \in I} B_{r_i}(0) = \setv{(i,v)}{i \in I, v \in B_{r_i}(0)}, \]
\[ \mcl{G}^\msc{U} = \bigsqcup_{i,j \in I} \setv{(v_i,v_j) \in B_{r_i}(0) \times B_{r_j}(0)}{\exp_{p_i}(v_i) = \exp_{p_j}(v_j)}. \]
We will write elements of $\mcl{B}^\msc{U}$ as $v_i = (i,v)$ and arrows as $(v_i,v_j)$.  Note that we always have $v_i = \exp_{p_i}^{-1}(x)$ for some $x \in U_i$.

The structure maps of this groupoid are defined as follows:
\begin{itemize}
\item source: $s(v_i,v_j) = v_i$
\item target: $t(v_i,v_j) = v_j$
\item unit: $\mbf{1}(v_i) = (v_i,v_i)$
\item inverse: $(v_i,v_j)^{-1} = (v_j,v_i)$
\item composition: $(v_j,v_k) \cdot (v_i,v_j) = (v_i,v_k)$
\end{itemize}

Call the \'{e}tale Riemannian groupoid $\mcl{G}^\msc{U} \rightrightarrows \mcl{B}^\msc{U}$ the \textit{Riemannian exponential groupoid} with respect to the open cover $\msc{U}$ of $\mcl{M}$.

\begin{prop}\label{prop:rh-grpoid-setting}
The \textsc{hrf} on a manifold $(\mcl{M},g,\phi)$ and target manifold $(\mcl{N},h)$ becomes \textsc{hrf} on the $n$-dimensional Riemannian exponential groupoid $(\mcl{G}^\msc{U} \rightrightarrows \mcl{B}^\msc{U},g,\phi)$ associated to an open exponential cover $\msc{U}$ of $\mcl{M}$.
\end{prop}

\begin{proof} The map $\fn{\phi}{\mcl{M}}{\mcl{N}}$ induces a Lie groupoid morphism $\phi = (\phi_0,\phi_1)$ from $\mcl{G}^{\msc{U}} \rightrightarrows \mcl{B}^{\msc{U}}$ to the trivial groupoid $\mcl{N} \rightrightarrows \mcl{N}$.  It is defined by
\[ \begin{aligned} 
\phi_0(v_i)     &= \phi(\exp_{p_i}(v_i)), \\
\phi_1(v_i,v_j) &= \phi(\exp_{p_i}(v_i)) = \phi(\exp_{p_j}(v_j)). \end{aligned} \]
Thus we can write $\phi_0 = \phi_1 = \exp^* \phi$ for these induced maps.  Note also that we could have defined them as
\[ \begin{aligned}
\phi_0(v_i)     &= \phi_0(\exp_{p_i}^{-1}(x)) = \phi(x), \\
\phi_1(v_i,v_j) &= \phi_1(\exp_{p_i}^{-1}(x),\exp_{p_j}^{-1}(x)) = \phi(x). \end{aligned} \]

It is easy to check that these maps are compatible with the structure maps of both groupoids.  That is, the following diagram is commutative.
\[ \xymatrixcolsep{3pc}  \xymatrixrowsep{2.5pc} \xymatrix{
\mcl{B}^\msc{U} \ar[d]_{\phi_0} \ar@{.>}[r]^{\mbf{1}_\mcl{G}} & \mcl{G}^\msc{U} \ar[d]^{\phi_1} \ar@/_1pc/[l]_{s_\mcl{G}} \ar@/^1pc/[l]_{t_\mcl{G}}   \\
\mcl{N}                         \ar@{.>}[r]^{\mbf{1}_\mcl{N}} & \mcl{N}                         \ar@/_1pc/[l]_{s_\mcl{N}} \ar@/^1pc/[l]_{t_\mcl{N}}
} \]
 
The main question is the $\mcl{G}$-invariance of $\nabla \phi_0 \otimes \nabla \phi_0$.  Let $\mcl{U}_i \subset \mcl{M}$ have coordinates $(x^i)$, and let a neighborhood $\mcl{V}_i$ of $\phi(p_i)$ have coordinates $(y^\alpha)$.  Then $B_{r_i}(0) \subset \mcl{B}^{\msc{U}}$ has coordinates $(z^i)$, where
\[ z^i = \exp_{p_i}^* x^i = x^i \circ \exp_{p_i}, \]
and a coframe on $T B_{R_i}(0)$ is $dz^i$, where
\[ dz^i = \exp_{p_i}^* dx^i = d(x^i \circ \exp_{p_i}). \]
To understand invariance, we must understand bisections of $\mcl{G}^\msc{U} \rightrightarrows \mcl{B}^\msc{U}$.  Let  $\sigma$ be a bisection, say
\fndef{\sigma}{B_{r_i}(0)}{\mcl{G}^\msc{U}}{v_i}{\big( \sigma_1(v_i),\sigma_2(v_i) \big)}
Since it is a bisection, we have $s \circ \sigma = \id_{\mcl{B}^{\msc{U}}}$, and this implies $\sigma_1 = \id_{\mcl{B}^{\msc{U}}}$.  Therefore we write
\[ \sigma(v_i) = \big( v_i,\tilde{\sigma}(v_i) \big), \]
where $\tilde{\sigma}(v_i)$ satisfies 
\[ \exp_{p_i} \tilde{\sigma}(v_i) = \exp_{p_i} (v_i). \]
Now we see that
\[ (t \circ \sigma)(v_i) = t \big( v_i,\tilde{\sigma}(v_i) \big) = \tilde{\sigma}(v_i), \]
or $t \circ \sigma = \tilde{\sigma}$.  

%With this, the invariance of each 1-form $dz^i$ is immediate:
%\begin{align*}
%(t \circ \sigma)^* dz^i
%&= d(z^i \circ t \circ \sigma) \\
%&= d(x^i \circ \exp_{p_i} \circ \tilde{\sigma}) \\
%&= d(x^i \circ \exp_{p_i}) \\
%&= dz^i.
%\end{align*}

Now, the induced map $\fn{\phi_0}{\mcl{B}^\msc{U}}{\mcl{N}}$ has pushforward
\[ (\phi_0)_* \in C^\infty \big( T^*\mcl{B}^\msc{U} \otimes (\phi_0)^* T\mcl{N} \big), \]
so 
\[ (\phi_0)_* = \frac{\partial \phi_0^\alpha}{\partial x^i} \, dx^i \otimes \left( \frac{\partial}{\partial y^\alpha} \right)_{\phi_0} = d\phi_0^\alpha \otimes \left( \frac{\partial}{\partial y^\alpha} \right)_{\phi_0}. \]
In any $B_{r_i}(0)$, we have
\begin{align*}
(t \circ \sigma)^* d\phi_0^\alpha
&= d(\phi_0^\alpha \circ  t \circ \sigma) \\
&= d(\phi^\alpha \circ \exp_{p_i} \circ \tilde{\sigma}) \\
&= d(\phi^\alpha \circ \exp_{p_i}) \\
&= d\phi_0^\alpha.
\end{align*}
If $\fn{f_0}{B_{r_i}(0)}{\R}$ is smooth, locally it is of the form $f_0 = f \circ \exp_{p_i}$ for some $\fn{f}{\mcl{U}_i}{\R}$.  Then
\begin{align*}
(t \circ \sigma)_* \left( \frac{\partial}{\partial y^\alpha} \right)_{\phi_0} f_0
&=\tilde{\sigma}_* \left( \frac{\partial}{\partial y^\alpha} \right)_{\phi_0} f_0 \\
&= \partial_\alpha( f_0 \circ \tilde{\sigma}) \\
&= \partial_\alpha( f \circ \exp_{p_i} \circ \tilde{\sigma}) \\
&= \partial_\alpha( f \circ \exp_{p_i}) \\
&= \left( \frac{\partial}{\partial y^\alpha} \right)_{\phi_0} f_0.
\end{align*}
From this, we conclude that $\nabla \phi_0 (\phi_0)_*$ is a $\mcl{G}^\msc{U}$-invariant tensor.  

In general, a metric $h$ on $T\mcl{N}$ induces a metric $h_{\phi}$ on the pull-back bundle $\phi^* T\mcl{N}$, given by
\[ h_{\phi}(\xi,\eta) = h(\phi_* \xi,\phi_* \eta), \]
for all $\xi,\eta \in T\mcl{M}$.  In this way, we get a metric on $(\phi_0)^* T\mcl{N}$, and it is $\mcl{G}^\msc{U}$-invariant:
\begin{align*}
(t \circ \sigma)^* h_{\phi_0}(\xi,\eta)
&= h_{\phi_0}(\xi,\eta).
\end{align*}

Thus $\nabla \phi_0 \otimes \nabla \phi_0$ is a $(2,0)$-tensor on $\mcl{B}^\msc{U}$: 
\[ \nabla \phi_0 \otimes \nabla \phi_0 = (h_{\phi_0})_{\lambda \mu} \partial_i \phi_0^\lambda \partial_j \phi_0^\mu \, dz^i \otimes dz^j. \]
It is therefore $\mcl{G}^\msc{U}$-invariant, and \textsc{hrf} makes sense on $\mcl{G}^\msc{U} \rightrightarrows \mcl{B}^{\msc{U}}$.
\end{proof}

This proposition shows that this framework is at least non-vacuous.  Before completing the proof of Theorem \ref{thm:compactness}, we need a definition and a result of Lott.

\begin{defn}
Let $\{(\mcl{G}_k \rightrightarrows \mcl{B}_k,g_k,\phi_k,O_{x_k})\}$ be a sequence of pointed, $n$-di\-men\-sion\-al
Riemannian groupoids with maps into some fixed Riemannian manifold $(\mcl{N},h)$.  Let $\{(\mcl{G}_\infty \rightrightarrows \mcl{B}_\infty,g_\infty,\phi_\infty,O_{x_\infty})\}$ be a pointed Riemannian groupoid with map $\fn{\phi_\infty}{\mcl{B}_\infty}{\mcl{N}}$.  Let $J_1$ be the groupoid of $1$-jets of local diffeomorphisms of $\mcl{B}_\infty$.  We say that
\[ (\mcl{G}_k \rightrightarrows \mcl{B}_k,g_k,\phi_k,O_{x_k}) \longrightarrow (\mcl{G}_\infty \rightrightarrows \mcl{B}_\infty,g_\infty,\phi_\infty,O_{x_\infty}) \] 
in the \textit{pointed smooth topology} if for all $R>0$, the following hold.
\begin{itemize}
\item There are pointed diffeomorphisms $\fn{\Psi_{k,R}}{B_R(O_{x_\infty})}{B_R(O_{x_k})}$, defined for large $k$, so that
\[ \Psi_{k,R}^* g_k|_{B_R(O_{x_i})} \longrightarrow g_\infty|_{B_R(O_{x_\infty})}. \]
\[ \Psi_{k,R}^* \phi_k|_{B_R(O_{x_i})} \longrightarrow \phi_\infty|_{B_R(O_{x_\infty})}. \]
\item After conjugating by $\Psi_{k,R}$, the images of 
\[ s_k^{-1} \big( \overline{B_{R/2}(O_{x_k})} \big) \cap t_k^{-1} \big( \overline{B_{R/2}(O_{x_k})} \big) \]
converge in $J_1$ in the Hausdorff sense to the image of
\[ s_\infty^{-1} \big( \overline{B_{R/2}(O_{x_\infty})} \big) \cap t_\infty^{-1} \big( \overline{B_{R/2}(O_{x_\infty})} \big) \]
in $J_1$.
\end{itemize}
\end{defn}

The following is \cite{Lott2007}*{Proposition 5.8}.

\begin{thm}\label{thm:lott-conv}
Let $\{(\mcl{M}_k,g_k,O_k)\}$ be a sequence of pointed complete $n$-di\-men\-sion\-al Riemannian manifolds.  Suppose that for each $p \geq 0$ and $r > 0$, there is some $C_{p,r} < \infty$ such that for all $k$,
\[ \max_{B_R(O_i)} | \nabla^p \Rm_k |_{\infty} \leq C_{p,r}. \]
Then there is a subsequence of $\{(M_k,O_k)\}$ that converges to some pointed $n$-di\-men\-sion\-al Riemannian groupoid $(G_\infty \rightrightarrows B_\infty, g_\infty, O_{x_\infty})$ in the pointed smooth topology.
\end{thm}

Now we can complete the proof.

\begin{proof}[Proof of Theorem \ref{thm:compactness}]
As Lott mentions, there is very little difference between the proofs of Hamilton's original theorem and  \cite{Lott2007}*{Theorem 5.12}.  The same is true here.  Namely, using Theorem \ref{thm:muller-bbs}, we obtain uniform bounds on the derivatives of the curvatures, which allow us to use Theorem \ref{thm:lott-conv}.  This is a version of Theorem \ref{thm:ham-conv} for groupoids, and gives subsequential convergence at one time to a pointed Riemannian groupoid.  

To extend this to the whole time interval, we apply Lemma \ref{lem:tech} and a version of \ref{lem:aa}, which gives another convergent subsequence.  Hence we get a limiting metric and map, which together solve \textsc{hrf} on $\mcl{M}$. %\marginpar{should really check this more carefully}
\end{proof}

\begin{rem}
As in \cite{Lott2007}*{Section 5}, Theorem \ref{thm:compactness} implies that the space of pointed $n$-dimensional \textsc{hrf} solutions with $\sup _t t | \Rm[g(t)] |_\infty < C$ is relatively compact among all \textsc{hrf} solutions on \'etale Riemannian groupoids.  Let $\msc{S}_{n,C}$ be the closure of this space.  It is easy to see that the blowdown procedure from Example \ref{ex:blowdown} defines an $\R^+$-action on the compact space $\msc{S}_{n,C}$.
\end{rem}

\section{A detailed example of the flow}\label{sec:ex}

In this section we provide a example of \textsc{hrf} on the three-dimensional nilpotent Lie group $\Nil^3$.  Ricci flow has been studied extensively on this space, as it is relatively easy to solve the Ricci flow equations explicitly, see \cites{IsenbergJackson1992,KnopfMcLeod2001,Glick2008}.  Study of the flow on this space, and homogeneous spaces in general, is facilitated by the fact that, due to preservation of isometry groups, Ricci flow is a system of \textsc{ode}.  By carefully selecting a map from $\Nil^3$, we are able to reduce harmonic-Ricci flow to a similar \textsc{ode} system, and compare the corresponding asympotics with those for Ricci flow solutions.  

Recall on the group
\[ \Nil^3 \cong \setvbigl{\begin{pmatrix} 1 & x & z \\ 0 & 1 & y \\ 0 & 0 & 1 \end{pmatrix}}{x,y,z \in \R} \subset \SL_3 \R, \]
left-invariant solutions of Ricci flow of the form 
\begin{equation}\label{eq:metric}
g(t) = A(t) \, \theta^1 \otimes \theta^1 + B(t) \, \theta^2 \otimes \theta^2 + C(t) \, \theta^3 \otimes \theta^3,
\end{equation}
have the following asympototics:
\begin{equation}\label{eq:RF-soln}
\begin{aligned}
A(t) &\sim A_0 K^{-1/3} t^{1/3}, \\
B(t) &\sim B_0 K^{-1/3} t^{1/3}, \\
C(t) &\sim C_0 K^{1/3} t^{-1/3},
\end{aligned}
\end{equation}
for the constant $K = A_0 B_0/3 C_0$.

We want to study \textsc{hrf} on $\Nil^3$.  Consider a function 
\[ \fnl{\phi}{\big( \Nil^3,g(t) \big)}{(\R,g_{\can})}, \]
and let $c=c(t) \geq 0$ be a non-increasing function.  For the resulting \textsc{hrf} system to remain a system of ordinary differential equations for the metric, we would like $\phi$ to be harmonic and $\nabla \phi \otimes \nabla \phi = d\phi \otimes d\phi$ to be a diagonal left-invariant tensor.  It is not hard to see that the latter condition requires that 
\[ \phi(x,y,z) = ax + by,  \]
for some $a, b \in \R$.  Note that such a function is also a group homomorphism, and that
\[ \tau_{g,g_{\can}} \phi = g^{ij}(\partial_i \partial_j \phi - \Gamma_{ij}^k \partial_k \phi) = 0, \]
so it is harmonic.  Then
\[ d\phi \otimes d\phi = a^2 \, dx \otimes dx + ab (dx \otimes dy + dy \otimes dx) + b^2 \, dy \otimes dy. \]
To keep the system diagonal, take $b = 0$, so that $d\phi \otimes d\phi = a^2 \, \theta^1 \otimes \theta^1$.  The \textsc{hrf} system is
\begin{equation}\label{eq:RH}
\begin{aligned}
\frac{d}{dt} A &= \frac{C}{B} + 2 a^2 c, \\
\frac{d}{dt} B &= \frac{C}{A}, \\
\frac{d}{dt} C &= -\frac{C^2}{AB}.
\end{aligned}
\end{equation}

Let us first make a few general observations about the long-time behavior of $A$, $B$, and $C$.  Set $f(t) = 2a^2 c(t)$ for simplicity.  Note that $\Phi = BC = B_0 C_0$ is conserved, $A$ and $B$ are increasing, and $C$ is decreasing.  This implies
\[ C' = -\frac{C^2}{AB} \geq - \frac{1}{A_0 B_0} C^2, \]
and integrating tells us that
\[ 0 < \frac{A_0 B_0 C_0}{A_0 B_0 + C_0 t} \leq C(t) \leq C_0, \]
for $t \geq 0$.  We conclude that $C(t) \rightarrow C_\infty \in [0,C_0)$ as $t \rightarrow \infty$.  Similarly, we see that
\[ C' = - \frac{C^3}{\Phi A} \geq - \frac{1}{\Phi A_0} C^3, \]
which implies
\[ 0 < \frac{A_0 B_0 C_0^2}{A_0 B_0 + 2 C_0^2 t} \leq C(t)^2 \leq C_0^2. \]
This gives 
\begin{equation}\label{eq:C2-lim}
\int_0^t C(s)^2 \, ds \geq \int_0^t \frac{A_0 B_0 C_0^2}{A_0 B_0 + 2 C_0 s} ds \longrightarrow \infty 
\end{equation}
as $t \rightarrow \infty$.  

Next we use $\Phi$ to see that
\begin{equation}\label{eq:Aprime-var}
A' = \frac{C^2}{\Phi} + f = \frac{\Phi}{B^2} + f,
\end{equation}
which we integrate to obtain
\begin{equation}\label{eq:A} 
A(t) = A_0 + \frac{1}{\Phi} \int_0^t C(s)^2 \, ds + \int_0^t f(s) \, ds.
\end{equation}
By \eqref{eq:C2-lim}, we have $A(t) \rightarrow \infty$ as $t \rightarrow \infty$, and we have a bound on the growth of $A$:
\begin{equation}\label{eq:A-bound}
A(t) \leq A_0 + \frac{C_0^2}{\Phi} \int_0^t ds + f_0 \int_0^t \, ds \leq A_0 + \left( \frac{C_0}{B_0} + f_0 \right) t.
\end{equation}
This implies
\begin{equation}\label{eq:intA-lim}
\int_0^t \frac{ds}{A(s)} \geq \int_0^t \frac{ds}{A_0 + \left( \frac{C_0}{B_0} + f_0 \right) s} \longrightarrow \infty 
\end{equation}
as $t \rightarrow \infty$.  

Finally,
\[ (B^2)' = \frac{2 \Phi}{A}, \]
which implies
\begin{equation}\label{eq:Bsquared}
B(t)^2 = B_0^2 + 2 \Phi \int_0^t \frac{ds}{A(s)},
\end{equation}
so, by \eqref{eq:intA-lim}, $B(t) \rightarrow \infty$ and $C(t) = \Phi/B(t) \rightarrow 0$ as $t \rightarrow \infty$.

\subsection{Constant coupling function}

Let us now consider the case when $c$ (and therefore $f$, which we write as $f_0$) is a constant.  From \eqref{eq:Aprime-var} we compute that
\[ \lim_{t \rightarrow \infty} \frac{A(t)}{f_0 t} \stackrel{LH}{=} \lim_{t \rightarrow \infty} \frac{A'(t)}{f_0}
= \lim_{t \rightarrow \infty} \left( \frac{C^2}{f_0 \Phi} + 1 \right) = 1, \]
so $A(t) \sim f_0 t$.  Using \eqref{eq:Bsquared} and $A \sim f_0 t \sim f_0 (t+1)$, we have
\[ B^2 \sim 2 \Phi \int_0^t \frac{ds}{A(s)} \sim \frac{2 \Phi}{f_0} \int_0^t \frac{ds}{s + 1} \sim \frac{2 \Phi}{f_0} \log t. \]
This gives the following.

\begin{prop}\label{prop:RH-const}
Solutions of \textsc{hrf} on $\Nil^3$ of the form \eqref{eq:metric} with map $\phi(x,y,z) = ax$ and $c>0$ constant have the following asympototics:
\begin{equation}\label{eq:RH-const-soln}
\begin{aligned}
A(t) &\sim 2a^2 c t, \\
B(t) &\sim \sqrt{\frac{B_0 C_0}{a^2 c} \log t}, \\
C(t) &\sim 2 \sqrt{\frac{a^2 c B_0 C_0}{\log t}}.
\end{aligned}
\end{equation}
\end{prop}

Note that if we attempt to take a limit of these solutions as $f_0 \rightarrow 0$, they \textit{do not} converge in a naive sense to the solutions of Ricci flow from \eqref{eq:RF-soln}.  To explain this, we examine certain coupling functions that decay as $t \rightarrow \infty$, and which yield behavior similar to that for Ricci flow.

\subsection{Nonconstant coupling function}

Now consider a coupling function such that
\[ c(t) \sim \frac{1}{t^r}, \]
where $r \geq 1$.  We make the ansatz that $A(t) \sim \alpha t^p$, for some $\alpha,p > 0$ to be determined.  From \eqref{eq:A-bound}, it is consistent to assume that $0 < p \leq 1$.  Then using \eqref{eq:A},
\begin{equation}\label{eq:A-limit}
\lim_{t \rightarrow \infty} \frac{A(t)}{\alpha t^p} 
\stackrel{LH}{=} \lim_{t \rightarrow \infty} \frac{A'(t)}{p \alpha t^{p-1}}
= \lim_{t \rightarrow \infty} \frac{1}{p \alpha t^{p-1}} \left( \frac{\Phi}{B^2} + \frac{2a^2}{t^r} \right).
\end{equation}
Finding this limit comes down to analyzing two limits:
\begin{equation}\label{eq:term1}
\lim_{t \rightarrow \infty} \frac{1}{B^2 t^{p-1}}, 
\end{equation}
\begin{equation}\label{eq:term2}
\lim_{t \rightarrow \infty} \frac{1}{t^{r+p-1}}. 
\end{equation}
Since $r \geq 1$ implies that \eqref{eq:term2} is zero for any $p > 0$, we need that
\[ B \sim \beta t^{\frac{1-p}{2}} \]
for some $\beta > 0$.  To find $\beta$, consider
\[ B^2 \sim 2 \Phi \int_0^t \frac{ds}{A(s)} \sim \frac{2 \Phi}{\alpha} \int_0^t \frac{ds}{(s + 1)^p} \sim \frac{2 \Phi}{\alpha(1-p)} t^{1-p}. \]
This now implies
\[ \lim_{t \rightarrow \infty} \frac{1}{B^2 t^{p-1}} = \frac{\alpha(1-p)}{2 \Phi}, \]
and so
\[ 1 \stackrel{?}{=} \lim_{t \rightarrow \infty} \frac{A(t)}{\alpha t^p} 
= \frac{\Phi}{p \alpha} \lim_{t \rightarrow \infty} \frac{1}{B^2 t^{p-1}}
= \frac{\Phi}{p \alpha} \frac{\alpha(1-p)}{2 \Phi}
= \frac{1-p}{2p}. \]
For $A(t) \sim \alpha t^p$ we therefore need $p=1/3$.  From here, we obtain the asymptotic behavior.  Modulo constants, it is that of the Ricci flow solutions \eqref{eq:RF-soln}.

\begin{prop}\label{prop:RH-r}
Solutions of \textsc{hrf} on $\Nil^3$ of the form \eqref{eq:metric} with map $\phi(x,y,z) = ax$ and $c \sim 1/t^r$, $r \geq 1$, have the following asympototics:
\begin{equation}\label{eq:RH-r-soln}
\begin{aligned}
A &\sim \alpha t^{1/3}, \\
B &\sim \sqrt{\frac{3\Phi}{\alpha}} t^{1/3}, \\
C &\sim \sqrt{\frac{\alpha \Phi}{3}} t^{-1/3}, 
\end{aligned}
\end{equation}
for some constant $\alpha$ depending on $r$ and the initial data.
\end{prop}

The reason that the limit as $f_0 \rightarrow 0$ of the solutions \eqref{eq:RH-const-soln}  is not the Ricci flow solutions \eqref{eq:RF-soln} lies in the integrability of the coupling function.  Informally, the lack of such a limit results from
\[ 0 = \lim_{t \rightarrow \infty} \lim_{f_0 \rightarrow 0} \int_0^t f_0 \, ds \ne \lim_{f_0 \rightarrow 0} \lim_{t \rightarrow \infty} \int_0^t f_0 \, ds = \infty. \]
To be more precise, consider $A(t)$ as given in \eqref{eq:A}.  When $r > 1$, $f(t)$ is integrable, allowing
\[ \int_0^t C(s)^2 \, ds \] 
to dominate and produce growth like $t^{1/3}$.  When $r<1$, $f(t)$ is not integrable, and 
\[ \int_0^t f(s) \, ds \]
dominates to produce linear growth.  

In numerical simulations, $0 < r < 1$ appears to be a transitionary region where solutions have properties of both \eqref{eq:RH-const-soln} and \eqref{eq:RH-r-soln}.  We were unable to obtain the precise asymptotics, but we expect that letting $r \rightarrow 0$ should recover \eqref{eq:RH-const-soln} and letting $r \rightarrow \infty$ should recover \eqref{eq:RF-soln}. 

%%% -------------------------------------------------------------------
%%% bibliography
%%%-------------------------------------------------------------------


\end{document}